\documentclass[12pt]{article}

\usepackage[usenames]{color}
\usepackage{amssymb}
\usepackage{amsmath}
\usepackage{amsthm}
\usepackage{amsfonts}
\usepackage{amscd}
\usepackage{graphicx}

\usepackage{hyperref}

\usepackage{fullpage}
\usepackage{float}

\usepackage{graphics}
\usepackage{latexsym}
\usepackage{epsf}
\usepackage{breakurl}

\newcommand{\Z}{\mathbb{Z}}
\newcommand{\G}{\mathcal{G}}
\newcommand{\T}{\mathcal{T}}

\DeclareMathOperator{\SL}{SL}
\newcommand{\s}{\SL(2,\Z)}

\title{
	{Uniqueness conjecture for extended Markov numbers}
}
\author{Matty van Son}

\begin{document}

\maketitle

\theoremstyle{plain}
\newtheorem{theorem}{Theorem}
\newtheorem{corollary}[theorem]{Corollary}
\newtheorem{lemma}[theorem]{Lemma}
\newtheorem{proposition}[theorem]{Proposition}

\theoremstyle{definition}
\newtheorem{definition}[theorem]{Definition}
\newtheorem{example}[theorem]{Example}
\newtheorem{conjecture}[theorem]{Conjecture}

\theoremstyle{remark}
\newtheorem{remark}[theorem]{Remark}

\newcommand{\Q}{\mathbb Q}

\vskip .2 in

\section*{Introduction}

Triples of \textit{regular Markov numbers} are the solutions to the Diophantine equation
\[
x^2+y^2+z^2=3xyz.
\]
These \textit{Markov triples} are the subject of the uniqueness conjecture of regular Markov numbers, introduced by Frobenius in 1913~\cite{frobenius1913}.

\begin{conjecture}[Uniqueness conjecture~(Frobenius~1913)]
	Markov triples are uniquely defined by their largest element.
\end{conjecture}

\begin{example}
	Both $(1,5,2)$ and $(5,29,2)$ are Markov triples.
	By the theory of Markov numbers the number 5 appears in infinitely many Markov triples.
	The uniqueness conjecture states that the only Markov triple in which 5 is the largest element is $(1,5,2)$.
\end{example}

This conjecture is well studied, and shows up in many interesting areas.
We refer to the book by Aigner~\cite{aigner2015} for a general reference.
Although the conjecture is not proven, some cases are known, see for example~\cite{baragar1996,button2001}.

In this note we extend the uniqueness conjecture for graphs of general Markov numbers, and show that for certain graphs the extended uniqueness conjecture fails (Theorem~\ref{theorem: uniq break}).
To define these graphs we first show how regular Markov numbers may be obtained from a graph of sequences.

\section{Generalised uniqueness conjecture}
The author and O.~Karpenkov~\cite{karpenkov2018} showed an extension for regular Markov numbers.
We define this extension Subsection~\ref{subsection: general Markov numbers} and introduce the generalised uniqueness conjecture.
We develop the first counterexamples to the conjecture in Subsection~\ref{subsection: counterexamples}.

\subsection{Development of general Markov numbers} \label{subsection: general Markov numbers}
We start with definitions of continued fractions.
\begin{definition}
	Let $\alpha=\big(a_i\big)_{i=1}^n$ and $\beta=\big(b_i\big)_{i=1}^m$ be finite sequences of positive integers.
	The \textit{concatenation} of $\alpha$ and $\beta$ is $\alpha\oplus\beta=(a_1,\ldots,a_n,b_1,\ldots,b_m)$.
	We often shorten the notation $\alpha\oplus\beta$ to $\alpha\beta$.
	
	The \textit{continued fraction expansion of $\alpha$} is
	\[
	a_1+\cfrac{1}{\ddots+\cfrac{1}{a_n}}.
	\] 
	and is denoted by $[a_1;a_2:\ldots:a_n]$.
\end{definition}
Next we define an important notion in the study of Markov numbers and sequences.
\begin{definition}
	For a sequence of positive integers $(a_1,\ldots,a_n)$ let $c$ and $d$ be the integers with $\gcd(c,d)=1$ such that
	\[
	[a_1;a_2:\ldots:a_{n-1}]=\frac{c}{d}.
	\]
	Define the \textit{integer sine} of $\alpha$ to be $c$.
	We use the notation $\breve{K}(\alpha)=c$.
\end{definition}

\begin{remark}
	The term \textit{integer sine} comes from the study of integer geometry.
	We recommend the book on this topic by Karpenkov~\cite{karpenkov2013} for interested readers.
	
	To calculate the integer sine of a sequence one may evaluate the continued fraction, or a polynomial of elements of the sequence called the  \textit{continuant}.
	For an explanation of continuants see the book by Graham, Knuth, and Patashnik~\cite{graham1990}.
\end{remark}

We define a graph structure that is used to study Markov numbers.

\begin{definition}
	Define operations $\mathcal{L}$ and $\mathcal{R}$ on triples of finite sequences of positive integers by
	\[
	\begin{aligned}
	\mathcal{L}(\alpha,\gamma,\beta)&=(\alpha,\alpha\gamma,\gamma),\\
	\mathcal{R}(\alpha,\gamma,\beta)&=(\gamma,\gamma\beta,\beta).
	\end{aligned}
	\]
	Define a binary graph $\G(\alpha,\beta)$ with root $(\alpha,\alpha\beta,\beta)$, where two vertices $v$ and $w$ are connected by an edge $(v,w)$ if 
	\[
	w=\mathcal{L}(v)\quad \mbox{or} \quad w=\mathcal{R}(v).
	\]
\end{definition}

We define operations $\chi$ and $X$ to obtain a triple graph of positive integers from a graph of triple sequences.

\begin{definition}
	Let $\chi$ be the map acting on triples of sequences by
	\[
	\chi(\alpha,\gamma,\beta)=\big(\breve{K}(\alpha),\breve{K}(\gamma),\breve{K}(\beta)\big).
	\] 
	Define a map $X$ taking a triple graph of sequences $\G(\alpha,\beta)$ to a triple graph of integers, where vertices $v$ are mapped to $\chi(v)$, and edges $(v,w)$ are mapped to $\big(\chi(v),\chi(w)\big)$.
	We call the graph 
	\[
	\T\big((1,1),(2,2)\big)=X\Big(\G\big((1,1),(2,2)\big)\Big)
	\]
	the \textit{graph of regular Markov numbers}.
	The triples appearing as vertices in this graph are the solutions to the Diophantine equation
	\[
	x^2+y^2+z^2=3xyz.
	\]
\end{definition}

A more complete treatment of the relation between Markov numbers and triple graphs of positive integers, and also triple graphs of $\s$ matrices and binary quadratic forms, may be found in the papers~\cite{karpenkov2018,vanson2018}.

Let us collect some known results about the graph of regular Markov numbers.

\begin{proposition}
	\begin{itemize}
		\item[]
		\item[\emph{(}i\emph{)}] Every triple at a vertex of the graph $\T\big((1,1),(2,2)\big)$ is a Markov triple.
		\item[\emph{(}ii\emph{)}] The graph contains every possible Markov triple.
		\item[\emph{(}iii\emph{)}] The vertices $v=(a_1,M,a_2)$ with $a_1\leq M$ and $a_2\leq M$, and $w=(b_1,Q,b_2)$ with $b_1\leq Q$ and $b_2\leq Q$ are connected by an edge $(v,w)$ if and only if either
		\[
		w=(a_1,3Ma_1-a_2,M)\quad \mbox{or} \quad w=(M,3Ma_2-a_1,a_2).
		\]
		\item[\emph{(}iv\emph{)}] The Markov graph is a tree (no loops or double edges). 
	\end{itemize}
\end{proposition}

This proposition is a collection of classical results in the theory of regular Markov numbers.
One may find a proof of each statement in the books by Cusick~\cite{cusick1989} or Aigner~\cite{aigner2015}.

One may define triple graphs of integers in the same way as Markov numbers but with different sequences.
In this note we consider the graphs $\T\big((a,a),(b,b)\big)$ where $a$ and $b$ are positive integers and $a<b$.
We call this a \textit{graph of general Markov numbers}.
We have the analogue of the uniqueness conjecture for regular Markov numbers.
\begin{conjecture}[Uniqueness conjecture for general Markov numbers]

	Let $a$ and $b$ be positive integers with $a<b$.
	Then each triple of integers at a vertex of the graph of general Markov numbers $\T\big((a,a),(b,b)\big)$ is uniquely defined by it's largest element.
\end{conjecture}

\subsection{First counterexamples to the general uniqueness conjecture} \label{subsection: counterexamples}

We define certain graphs of general Markov numbers for which this conjecture is false. 
\begin{definition}
	For any positive integer $n$ define positive integers $a_n$ and $b_n$ by
	\[
	\begin{aligned}
	a_n&=n^2+3,\\
	b_n&=n^4+5n^2+5.
	\end{aligned}
	\] 
	Note that $\gcd(a_n,b_n)=1$, and that the ratio $b_n/a_n$ is equal to the continued fraction
	\[
	\frac{b_n}{a_n}=\left[ 1+n^2 ; 1 :  2+n^2 \right]
	\]

\end{definition}

\begin{definition}
	Define the sequences $S_n(0)$ and $S_n(1)$ by
	\[
	S_n(0)=(na_n,na_n),\ \ S_n(1)=(nb_n,nb_n).
	\]
	We notate the graphs given with these sequences by $\T_n=X\Big(\G\big(S_n(0),S_n(1)\big)\Big)$.
\end{definition}

\begin{example}
	We show the sequences $S_n(0)$ and $S_n(1)$ for $n=1,\ldots,5$, along with the continued fraction of $b_n/a_n$, in Table~\ref{table: sn}.
	
	\begin{table}
		\centering
		\begin{tabular}{|l|c|c|c|}
			\hline
			$n$ & $S_n(0)$ & $S_n(1)$ & $b_n/a_n$ \\ \hline
			1 & (4,4) & (11,11) & [2; 1: 3] \\
			2 & (14, 14) & (82, 82) & [5; 1: 6] \\
			3 & (36, 36) & (393, 393) & [10; 1: 11] \\
			4 & (76, 76) & (1364, 1364) & [17; 1: 18] \\
			5 & (140, 140) & (3775, 3775) & [26; 1: 27] \\\hline
		\end{tabular}
		\caption{Sequences $S_n(0)$ and $S_n(1)$ for $n=1,\ldots,5$.}
		\label{table: sn}
	\end{table}

\end{example}

We present the main result.
\begin{theorem} \label{theorem: uniq break}
	The uniqueness conjecture for general Markov numbers does not hold for any graph $\T_n$, where $n$ is a positive integer.
	
\end{theorem}

In the graph of Markov numbers $\T_n$ for any $n>0$ there are triples defined
\[
\begin{aligned}
&\Big( \breve{K}\big(S_n(0)\big),\quad \breve{K}\big( S_n(0)^{5j+1} S_n(1) \big) ,\quad \breve{K}\big(S_n(0)^{5j}S_n(1)\big) \Big),\\
&\Big( \breve{K}\big(S_n(0) S_n(1)^{3j}\big),\quad \breve{K}\big( S_n(0) S_n(1)^{3j+1} \big) ,\quad \breve{K}\big(S_n(1)\big) \Big),
\end{aligned}
\]
for all $j\geq 1$.
We show that the largest element of these triples are equal, and hence the uniqueness conjecture for general Markov numbers fails for the graphs $\T_n$.
More specifically we have the following proposition.
\begin{proposition} \label{proposition: uniq break}
	For all positive integers $n$ and $j$ we have that
	\[
	\breve{K} \big( S_n(0)^{5j+1} S_n(1) \big) = \breve{K} \big( S_n(0) S_n(1)^{3j+1} \big).
	\]
\end{proposition}

\begin{example}
	
	The simplest examples are in the graph $\T_1$, which contains the triples
	\[
	\begin{aligned}
	\Big( \breve{K}(4,4),\ &\breve{K}\big( (4,4)^6 \oplus (11,11) \big) ,\ \breve{K}\big((4,4)^5\oplus (11,11)\big) \Big),\\
	\Big( \breve{K}\big((4,4)\oplus(11,11)^3 \big),\ &\breve{K}\big( (4,4)\oplus(11,11)^4\big) ,\ \breve{K}(11,11) \Big),
	\end{aligned}
	\]
	which, when evaluated, give
	\[
	\begin{aligned}
	(4,\, &355318099,\, 19801199),\\
	(2888956,\, &355318099,\, 11).
	\end{aligned}
	\]

\end{example}

\subsection{Proof of Theorem~\ref{theorem: uniq break}}
Theorem~\ref{theorem: uniq break} follows from Proposition~\ref{proposition: uniq break}.
To prove this proposition we first define sequences of positive integers $\big(L_n(j)\big)_{j>0}$ and $\big(R_n(j)\big)_{j>0}$ containing the values
\[
\breve{K} \big( S_n(0)^{5j+1} S_n(1) \big)\quad\mbox{and}\quad  \breve{K} \big( S_n(0) S_n(1)^{3j+1} \big).
\] 
We show in Lemmas~\ref{lemma: R is a subsequence of A} and~\ref{lemma: L is a subsequence of A} that both $\big(L_n(j)\big)_{j>0}$ and $\big(R_n(j)\big)_{j>0}$ are subsequences of another sequence $\big(A_n(j)\big)_{j>0}$ for every $n>0$.
Then we show that their elements align within $\big(A_n(j)\big)_{j>0}$ in such a way that Proposition~\ref{proposition: uniq break} holds.

\begin{definition}
	Let $n$ be a positive integer and let $a_n=n^2+3$ and $b_n=n^4+5n^2+5$.
	Define 
	\[
	\begin{aligned}
	l_n&=(na_n)^2+2, \quad r_n=(nb_n)^2+2,\\
	L_n(1)&=\breve{K}(na_n,na_n,nb_n,nb_n), \quad L_n(2)=\breve{K}(na_n,na_n,na_n,na_n,nb_n,nb_n),\\
	R_n(1)&=\breve{K}(na_n,na_n,nb_n,nb_n), \quad R_n(2)=\breve{K}(na_n,na_n,nb_n,nb_n,nb_n,nb_n).
	\end{aligned}
	\]
	For $j>2$ define
	\[
	L_n(j)=l_nL_n(j-1)-L_n(j-2) \quad \mbox{and} \quad R_n(j)=r_nR_n(j-1)-R_n(j-2).
	\]
\end{definition}

We relate the sequences $\big(L_n(j)\big)_{j>0}$ and $\big(R_n(j)\big)_{j>0}$ to the numbers in Theorem~\ref{theorem: uniq break}.
\begin{proposition}
	The following statements are equivalent:
	\begin{itemize}
		\item[\emph{(}i\emph{)}] For all positive integers $n$ and $i$ we have that
		\[
		\breve{K} \big( S_n(0)^{5i+1}\oplus S_n(1) \big) = \breve{K} \big( S_n(0)\oplus S_n(1)^{3i+1} \big).
		\]
		\item[\emph{(}ii\emph{)}] For all positive integers $n$ and $i$ we have that
		\[
		L_n(5i+1) = R_n(3i+1).
		\]
	\end{itemize}
\end{proposition}
The proof of this proposition relies on the recurrence relation for general Markov numbers which may be found in~\cite[Theorem~7.15]{karpenkov2018}.
\begin{example}
	The first $6$ elements of the sequences $\big(L_1(j)\big)_{j>0}$ and $\big(R_1(j)\big)_{j>0}$ are
	\[
	\begin{aligned}
	\big(L_1(j)\big)_{j=1}^{6}&=(191, 3427, 61495, 1103483, 19801199, 355318099),\\
	\big(R_1(j)\big)_{j=1}^{6}&=(191, 23489, 2888956, 355318099, 43701237221, 5374896860084).
	\end{aligned}
	\]
\end{example}

Next we define the sequence $\big(A_n(j)\big)_{j>0}$.
\begin{definition}
	Let $A_n(1)=1$ and $A_n(2)=n(n^2+4)$.
	Then for $j>2$ define
	\[
	A_n(j)=nA_n(j-1)+A_n(j-2).
	\]
\end{definition}

\begin{remark}
	We guessed the structure of this sequence from looking at the case for $n=1$, where $\big(A_1(j)\big)_{j>0}$ is the sequence A022095 in~\cite{oeis}.
\end{remark}

\begin{lemma} \label{lemma: R is a subsequence of A}
	For all positive integers $n$ and $j$ we have that 
	\[
	R_n(j)=A_n(10j).
	\]
\end{lemma}

\begin{proof}
	For any $n>0$ we have  $na_n=3n+n^3$, $b=n^5+5n^3+5n$, and also that $r_n=(nb_n)^2+2$. 
	Hence
	\[
	R_n(1)=\breve{K}(na_n,na_n,nb_n,nb_n)=n^{11}+11n^9+44n^7+76n^5+51n^3+8n.
	\]
	By computation we see that $A_n(10)=R_n(1)$.
	Also we see that
	\[
	\begin{aligned}
	R_n(2)
	=&{n}^{21}+21\,{n}^{19}+189\,{n}^{17}+951\,{n}^{15}+2926\,{n}^{13}+5655\,{n}^{11}+\\
	&6787\,{n}^{9}+4818\,{n}^{7}+1827\,{n}^{5}+301\,{n}^{3}+13\,n
	=A_n(20).
	\end{aligned}
	\]
	This serves as a base of induction.
	
	Assume that $A_n(10j)=R_n(j)$ for all $j=1,\ldots,k{-}1$, for some $k>2$.
	Then
	\[
	\begin{aligned}
	R_n(k)&=r_nR_n(k-1)-R_n(k-2)\\
	&=r_nA_n\big(10(k-1)\big)-A_n\big(10(k-2)\big),\\
	&=r_nA_n(10k-10)-A_n(10k-20),
	\end{aligned}
	\]
	with the first equality following definition, and the second equality following the induction hypothesis.
	Note that
	\[
	\begin{aligned}
	A_n(10k-10)&=nA_n(10k-10-1)+A_n(10k-10-2)\\
	&=(n^2+1)A_n(10k-10-2)+A_n(10k-10-3)\\
	&\vdots\\
	&=x_1A_n(10k-10-9)+x_2A_n(10k-10-10),
	\end{aligned}
	\]
	where $x_1$ and $x_2$ are given through direct computation (we used Maple software) by
	\[
	\begin{aligned}
	x_1&={n}^{9}+8\,{n}^{7}+21\,{n}^{5}+20\,{n}^{3}+5\,n,\\
	x_2&={n}^{8}+7\,{n}^{6}+15\,{n}^{4}+10\,{n}^{2}+1.
	\end{aligned}
	\]
	In the same manner we have
	\[
	\begin{aligned}
	A_n(10k)&=nA_n(10k-1)+A_n(10k-2)\\
	&=(n^2+1)A_n(10k-2)+A_n(10k-3)\\
	&\vdots\\
	&=y_1A_n(10k-19)+y_2A_n(10k-20),
	\end{aligned}
	\]
	where $y_1$ and $y_2$ are given through direct computation by
	\[
	\begin{aligned}
	y_1=&{n}^{19}+18\,{n}^{17}+136\,{n}^{15}+560\,{n}^{13}+1365\,{n}^{11}+\\
	&2002
	\,{n}^{9}+1716\,{n}^{7}+792\,{n}^{5}+165\,{n}^{3}+10\,n
	,\\
	y_2=&{n}^{18}+17\,{n}^{16}+120\,{n}^{14}+455\,{n}^{12}+1001\,{n}^{10}+\\
	&1287
	\,{n}^{8}+924\,{n}^{6}+330\,{n}^{4}+45\,{n}^{2}+1
	.
	\end{aligned}
	\]
	Using this information we have that
	\[
	\begin{aligned}
	R_n(k)&=r_nA_n(10k-10)-A_n(10k-20)\\
	&=r_n\big(x_1A_n(10k-10-9)+x_2A_n(10k-10-10)\big)-A_n(10k-20)\\
	&=r_nx_1A_n(10k-19)+(r_nx_2-1)A_n(10k-20).
	\end{aligned}
	\]
	Through direct computation we see that $y_1=r_nx_1$ and $y_2=r_nx_2-1$.
	Hence we have that
	\[
	R_n(k)=A_n(10k),
	\]
	so the induction holds and the proof is complete.
\end{proof}
Lemma~\ref{lemma: R is a subsequence of A} says that $\big(R_n(j)\big)_{j>0}$ is a subsequence of $\big(A_n(j)\big)_{j>0}$.
Now we show an analogous statement for $\big(L_n(j)\big)_{j>0}$.

\begin{lemma} \label{lemma: L is a subsequence of A}
	For all positive integers $n$ and $i$ we have that 
	\[
	L_n(j)=A_n\big(10+6(j-1)\big).
	\]
\end{lemma}

\begin{proof}
	We use induction.
	We have that $A_n(10)=L_n(1)$, as in the proof for Lemma~\ref{lemma: R is a subsequence of A} .
	We also have $l_n=(n^3+3n)^2{+}2$, and that
	\[
	\begin{aligned}
	A_n(16)&={n}^{17}+17{n}^{15}+119{n}^{13}+441{n}^{11}+925{n}^{9}+1086{
		n}^{7}+658{n}^{5}+169{n}^{3}+11n\\
	&=L_n(2).
	\end{aligned}
	\]
	This is our base of induction.
	Assume that $A_n\big(10+6(j-1)\big)=A_n(6j+4)=L_n(j)$ for all $j=1,\ldots,k{-}1$, for some $k>2$.
	Then
	\[
	\begin{aligned}
	L_n(k)&=l_nL_n(k-1)-L_n(k-2)\\
	&=l_nA_n(6k-2)-A_n(6k-8),
	\end{aligned}
	\]
	with the first equality following definition, and the second equality following the induction hypothesis.
	In a similar way to the proof for $\big(R_n(j)\big)_{j>0}=\big(A_n(j)\big)_{j>0}$ we have that
	\[
	A_n(6k-2)=w_1A_n(6k-7)+w_2A_n(6k-8),
	\]
	where $w_1=n^5+4n^3+3n$ and $w_2=n^4+3n^2+1$.
	Also we have
	\[
	A_n\big(10+6(k-1)\big)=A_n(6k+4)=z_1A_n(6k-7)+z_2A_n(6k-8),
	\]
	where
	\[
	\begin{aligned}
	z_1&={n}^{11}+10\,{n}^{9}+36\,{n}^{7}+56\,{n}^{5}+35\,{n}^{3}+6\,n,\\
	z_2&={n}^{10}+9\,{n}^{8}+28\,{n}^{6}+35\,{n}^{4}+15\,{n}^{2}+1.
	\end{aligned}
	\]
	Now we have that
	\[
	\begin{aligned}
	L_n(k)&=l_nA_n(6k-2)-A_n(6k-8),\\
	&=l_nw_1A_n(6k-7)+(l_nw_2-1)A_n(6k-8),
	\end{aligned}
	\]
	and by direct computation we see that $z_1=l_nw_1$ and $z_2=l_nw_2-1$.
	
	So $L_n(k)=A_n\big(10+6(k-1)\big)$ and induction holds.
	This completes the proof.
\end{proof}

We prove Proposition~\ref{proposition: uniq break}.
\begin{proof}[Proof of Proposition~\ref{proposition: uniq break}]
	Given Lemmas~\ref{lemma: R is a subsequence of A} and~\ref{lemma: L is a subsequence of A} we need only show the values $L_n(5j+1)$ and $R_n(3j+1)$ align within the sequence $\big(A_n(j)\big)_{j>0}$.
	Indeed we have that
	\[
	\begin{aligned}
	L_n(5j+1)&=A_n\big(10+6(5j+1-1)\big)=A_n(30j+10),\\
	R_n(3j+1)&=A_n\big(10(3j+1)\big)=A_n(30j+10).
	\end{aligned}
	\]
	Hence the claim is proved.
\end{proof}

Theorem~\ref{theorem: uniq break} follows as a corollary.

\bibliographystyle{plain}
\bibliography{biblio_ALL}

\end{document}